\documentclass{amsart}
\usepackage{amsmath,amsthm,amssymb}
\usepackage{xcolor}
\usepackage{graphicx}
\usepackage{tikz-cd}
\usepackage{paralist}
\usepackage{environ}
\usepackage{xcolor}
\usepackage{hyperref}
\usepackage{centernot}
\usepackage{mathtools}
\usepackage{marvosym}
\usepackage{xcolor,cancel}
\usepackage{float}
\usepackage{subcaption}
\usepackage{mathrsfs}
\usepackage{verbatim}
\usepackage{csquotes}
\usepackage{pstricks}
\usepackage{mdframed}
\usepackage[most]
{tcolorbox}

\makeatletter
\def\square{\pst@object{square}}
\def\square@i(#1,#2)#3{{\use@par\solid@star\psframe[origin={#1,#2}](#3,#3)}}
\makeatother

\makeatletter
\DeclareFontFamily{U}{tipa}{}
\DeclareFontShape{U}{tipa}{bx}{n}{<->tipabx10}{}
\newcommand{\arc@char}{{\usefont{U}{tipa}{bx}{n}\symbol{62}}}%

\newcommand{\arc}[1]{\mathpalette\arc@arc{#1}}

\newcommand{\arc@arc}[2]{%
	\sbox0{$\m@th#1#2$}%
	\vbox{
		\hbox{\resizebox{\wd0}{\height}{\arc@char}}
		\nointerlineskip
		\box0
	}%
}
\makeatother

\makeatletter
\newcommand{\doublewedge}{\big@doubleop{\wedge}}
\newcommand{\big@doubleop}[1]{%
	\DOTSB\mathop{\mathpalette\big@doubleop@aux{#1}}\slimits@
}

\newcommand\big@doubleop@aux[2]{%
	\sbox\z@{$\m@th#1#2$}%
	\makebox[1.35\wd\z@][s]{$\m@th#1#2\hss#2$}%
}

\newcommand{\abs}[1]{\left|#1\right|}     

\theoremstyle{plain}
\newtheorem{theorem}{Theorem}
\newtheorem{lemma}{Lemma}

\newtheorem{definition}{Definition}

\newtheorem{example}{Example}
\newtheorem{corollary}{Corollary}
\newtheorem{proposition}{Proposition}

\usepackage{cancel}
 
\usepackage[title]{appendix}

\begin{document}
	\title{Proximal and Descriptive Topological Rings: Structure and Properties}
	\author[M. Almahariq]{Maram Almahariq}
	\address{Department of Mathematics, Birzeit university, Ramallah, Palestine,
	}
	\email{maram.mahareeq.14@gmail.com}
    

\maketitle
\begin{abstract}
The set of open subsets in a proximity (descriptive) space
$(\mathfrak{X},\zeta)$ forms a topological system that allows the definition of proximal and descriptive topological groups. This framework is not restricted to proximal (descriptive) groups, but also extends naturally to proximal (descriptive) topological rings. In this study, I review the key properties and related results, with particular emphasis on the definition of proximal (descriptive) fields. Moreover, I address the behavior of closed sets in proximal (descriptive) topological rings.
\end{abstract}

	\maketitle
	\tableofcontents

{
\section{Introduction}
In this paper, I introduce the concepts of proximal rings and descriptive rings, along with their basic properties. The first section provides the necessary background, including the main definitions and preliminary results that form the basis of the study. The second section presents the main results, focusing on the classification of topological rings from both proximal and descriptive perspectives. Finally, the paper concludes with a summary, outlining potential directions for future research. 
\section{Preliminaries}
Topology is one of the most important branches of mathematics, and Hausdorff made the first definition of it in 1914. His definition considers the elements of a nonvoid set 
$\mathfrak{X}$ as points and associates with each point 
$\mathfrak{x} \in \mathfrak{X}$ a collection of subsets, called the neighborhoods of 
$\mathfrak{x}$, determined by a function 
$\mathfrak{N}$. These neighborhoods are required to satisfy certain axioms, which together ensure that $\mathfrak{X}$ forms a topological space \cite{Hausdorff1914}.

Algebraic structures are defined as nonvoid sets equipped with one or two binary operations, the most common of which are addition and multiplication. Among the most important algebraic structures, which find applications beyond algebra itself, are groups and rings. A set
$\mathfrak{G}$ forms a group under a binary operation if it satisfies the following properties: closure, associativity, the existence of an identity element, and the existence of an inverse element for each member of the set. A ring is an algebraic structure built upon a set that is already a group under one operation (typically addition), while the second operation (multiplication) forms a semigroup and satisfies the distributive law with respect to the first operation (\cite{Cayley1854}, see also \cite{Hilbert1897}).

The concept of computational proximity was originally based on metric spaces, where a gap function is defined to measure how close or far two subsets are within a nonvoid set by determining the distance between points in these subsets (\cite{PetersOzturk2025,PetersComputational2016}, see also \cite{PetersGuadagni2016}). However, in 1936, \v{C}ech introduced the concept of a proximity space (abbreviated as ($\textbf{pro.space}$)) by defining a proximity relation 
$\zeta$ that satisfies a specific set of axioms \cite{Cech1966}.  Later, researchers such as Lodato and Efremovi\v{c} extended \v{C}ech’s axioms by introducing additional conditions, resulting in what is now known as the Lodato relation and the Efremovi\v{c} relation, respectively \cite{Lodato1962,Efremovic1952}. \newline\newline
The closure of a subset $\mathfrak{A}$ of a $\textbf{pro.space}$ $(\mathfrak{X},\zeta)$ is defined as follows:
\[cl_{\zeta}(\mathfrak{A})
= \{ \mathfrak{x} \in \mathfrak{X} \mid \{\mathfrak{x}\} \,\zeta\, \mathfrak{A} \}
\quad \text{(see \cite{Peters2014})}\]

\begin{definition} \cite{Puisseux1850,Efremovic1952}
Let $\alpha : \mathfrak{X} \mapsto \mathfrak{X}$ be a map defined on a $\textbf{pro.space}$ $(\mathfrak{X},\zeta)$. The map $\alpha$ is called a proximal continuous map (abbreviated as ($\textbf{pro.con}$)), for $\mathcal{W},\mathcal{K} \subseteq \mathfrak{X}$ 
\[\mathcal{W} \;\zeta\; \mathcal{K} \;\longrightarrow\; 
\alpha(\mathcal{W}) \;\zeta\; \alpha(\mathcal{K})\]
\end{definition}

It can be observed that the composition of two $\textbf{pro.con}$ maps is also $\textbf{pro.con}$. \cite{PetersVergili2023} \newline

A map $\alpha$ is called a proximal isomorphism (abbreviated as \textbf{(pro.iso)}) if it is \textbf{(pro.con)} and its inverse is also \textbf{(pro.con)}. Furthermore, if $\alpha$ is bijective, it is referred to as a proximal homeomorphism (abbreviated as \textbf{(pro.homo)}) \cite{PetersVergili:2020}. \newline

Let $\zeta_\Phi$, where $\Phi$ is a probe function, be a proximity relation based on descriptive quantities such as color, shape, texture, etc. In this case, this relation is called a descriptive relation, and the space consisting of a nonvoid set $\mathfrak{X}$ and $\zeta_\Phi$, $(\mathfrak{X}, \zeta_\Phi)$ is called a descriptive space \textbf{(des.space)}, where $\zeta_\Phi$ satisfies the descriptive proximity axioms \cite{DiConcilio2018,Peters2013}.
\begin{definition} \cite{Is2024}
If $(\mathfrak{G},+)$ is a group equipped with a proximity relation $\zeta$. Then, $(\mathfrak{G},+,\zeta)$ is called a proximal group, denoted by ($\boldsymbol{P\mathcal{G}}_{\zeta}$), if the following maps are $\textbf{pro.con}$:

\begin{enumerate}
\item [1.] $\operatorname{add}_\zeta: \mathfrak{G} \times \mathfrak{G} \mapsto \mathfrak{G}, 
 \quad \operatorname{add}_\zeta(x_{0}, x_{1}) = x_{0} + x_{1}.$
 \item [2.]$ \operatorname{inv}_\zeta: \mathfrak{G} \mapsto \mathfrak{G}, \quad \operatorname{inv}_\zeta(x) = -x.$
 \end{enumerate}
\end{definition}
For a descriptive group, which is defined analogously to a proximal group but with a descriptive relation $\zeta_\Phi$, the corresponding maps $\operatorname{add}_{\zeta_\Phi}$ and $\operatorname{inv}_{\zeta_\Phi}$ are (\textbf{des.con}).\newline\newline
After presenting the concept of proximal group and its properties, it is natural that the following question is raised: Can we introduce a proximal topological ring (descriptive topological ring), which will be abbreviated as proximal rings (descriptive rings) ? Technically, I will express proximal rings with the symbol $\boldsymbol{P\mathcal{R}}_{\zeta}$ and descriptive rings with $\boldsymbol{P\mathcal{R}}_{\zeta_{\Phi}}$. The answer is yes, by introducing a ring equipped with the proximity relation $\zeta$ and the descriptive relation $\zeta _{\Phi}$, one obtains a topology $\boldsymbol{\tau_{\zeta}}$ ($\boldsymbol{\tau_{\zeta_{\Phi}}}$) induced by the open subsets of \textbf{pro.space} $(\mathcal{R},\zeta)$ and  \textbf{des.space} $(\mathcal{R},\zeta_{\Phi})$ whose complements are closed sets. This space is equipped with three $\textbf{pro.con}$ ($\textbf{des.con}$) maps, respectively,  which are defined as follows:

\section{Main result}

The results presented in this paper are inspired by studies in \cite{Warner1993,Chakravartty2023}, which discuss the concept of topological rings.

\subsection{Proximal Rings}
\begin{definition}
Consider the ring $(\mathcal{R},\boxplus,\boxtimes)$ equipped with a proximity relation $\zeta$. Then it is called a $\boldsymbol{P\mathcal{R}}_{\zeta}$, if the following maps are \textbf{pro.con}.
 \\ For $(w_{0}, w_{1}) \in \mathcal{R} \times \mathcal{R}$ and $w \in \mathcal{R}$, define: 
\begin{enumerate}

\item [1.] $\operatorname{add}_\zeta: \mathcal{R} \times \mathcal{R} \mapsto \mathcal{R}, 
 \quad \operatorname{add}_\zeta(w_{0}, w_{1}) = w_{0} \boxplus w_{1}. $
 
\item [2.] $\operatorname{mul}_\zeta: \mathcal{R} \times \mathcal{R} \mapsto \mathcal{R}, \quad  \operatorname{mul}_\zeta(w_{0}, w_{1}) = w_{0} \boxtimes w_{1}. $

\item [3.]$ \operatorname{inv}_\zeta: \mathcal{R} \mapsto \mathcal{R}, \quad \operatorname{inv}_\zeta(w) = -w.$
\end{enumerate}
\end{definition}

\begin{definition}
Let $\mathcal{W},\mathcal{K}$ be subsets of a $\boldsymbol{P\mathcal{R}}_{\zeta}$ $(\mathcal{R},\boxplus,\boxtimes)$. Then,
$\mathcal{W} \boxplus \mathcal{K}$, $\mathcal{W} \boxtimes \mathcal{K}$ and $\mathcal{-W}$ are given by $\{ w \boxplus k : w \in \mathcal{W}, k \in \mathcal{K}\}$ , $\{ w \boxtimes k : w \in \mathcal{W}, k \in \mathcal{K}\}$ and $\{ -w: w \in \mathcal{W}\}$, respectively. 
\end{definition}

\begin{example} \label{Ex: real number}
Let $\zeta$ be a proximity relation on $(\mathbb{R},+, \cdot)$, defined by 
\[
\mathcal{W} \;\zeta\; \mathcal{K} \;\iff\; \mathcal{W} \subseteq \mathcal{K}
\]
Then, $(\mathbb{R},+, \cdot, \zeta)$ is a $\boldsymbol{P\mathcal{R}}_{\zeta}$. \newline\newline  
 Assume that $\mathcal{W}_{1} \times \mathcal{K}_{1}$, $\mathcal{W}_{2} \times \mathcal{K}_{2}$ are subsets of $\mathbb{R} \times \mathbb{R}$ such that $(\mathcal{W}_{1} \times \mathcal{K}_{1})$ $\zeta$ $(\mathcal{W}_{2} \times \mathcal{K}_{2})$, provided $\mathcal{W}_{1} \subseteq \mathcal{W}_{2}$ and $\mathcal{K}_{1} \subseteq \mathcal{K}_{2}$. Hence, $w_{1} + k_{1} \in \mathcal{W}_{2} + \mathcal{K}_{2}$ where $w_{1} + k_{1}$ is an element of $\mathcal{W}_{1} + \mathcal{K}_{1}$. It follows that, $\operatorname{add}_\zeta ( \mathcal{W}_{1} \times \mathcal{K}_{1})$ $\zeta$ $\operatorname{add}_\zeta ( \mathcal{W}_{2} \times \mathcal{K}_{2})$. Similarly, one obtains that the multiplicative map is \textbf{pro.con}. 
Now, consider $\mathcal{W},\mathcal{K}$ to be subsets of $\mathbb{R}$ such that $\mathcal{W}$ $\zeta$ $\mathcal{K}$. Therefore, $\mathcal{W} \subseteq \mathcal{K}$. Moreover, $-\mathcal{W} \subseteq -\mathcal{K}$. This implies that additive inverse map is \textbf{pro.con}.
\end{example}

\begin{definition}
Suppose that $(\mathcal{R},\boxplus, \boxtimes,\zeta)$ is a $\boldsymbol{P\mathcal{R}}_{\zeta}$. If $S \subseteq \mathcal{R}$ is a subring , then $(S,\boxplus, \boxtimes,\zeta)$ is said to be a proximal subring.
\end{definition}
Note that, If $S$ is a subring of $\mathcal{R}$, the additive, multiplicative, and additive inverse maps on $S$ are \textbf{pro.con}, 
being the restrictions of the respective maps defined on $\mathcal{R}$.

It follows from example ~\ref{Ex: real number} that, $(\mathbb{Z},+,\cdot,\zeta)$ is a proximal subring of $(\mathbb{R},+,\cdot,\zeta)$.
\begin{lemma} \label{lem:boxplus}
Let $(\mathcal{R}, \boxplus, \boxtimes, \zeta)$ be a $\boldsymbol{P\mathcal{R}}_{\zeta}$ and $c$ is an element of $\mathcal{R}$, implies that 
\begin{enumerate}
\item  $\phi_{c} : \mathcal{R} \mapsto \mathcal{R}, \phi_{c}(w)=w\boxplus c.$

\item $ \beta_{c} : \mathcal{R} \mapsto \mathcal{R}, \beta_{c}(w)=c\boxplus w.$ 
\end{enumerate}
are \textbf{pro.homo} maps.
\end{lemma}

\begin{proof}
Define $\alpha_{c} : \mathcal{R} \mapsto \mathcal{R} \times \mathcal{R}$ by $\alpha_{c}(w) = (w,c)$. Hence, $\alpha_{c}$ is $\textbf{pro.con}$. Since if $\mathcal{W},\mathcal{K}$ are near subsets of $\mathcal{R}$. Therefore, $(\mathcal{W},\{c\})$ is close to $(\mathcal{K},\{c\})$. It follows that $\alpha_{c}$ is $\textbf{pro.con}$. It should be observed that, $\phi_{c} = \operatorname{add}_\zeta \circ \alpha_{c}$, provided $\phi_{c}$ is $\textbf{pro.con}$. \newline
Moreover, $\phi_{-c}$ is also a $\textbf{pro.con}$ map. \newline
It remains to prove that $\phi_{c}(\phi_{-c}(w)) = \phi_{-c}(\phi_{c}(w)) = w$, which indeed holds.
The argument for $\beta_{c}$ is identical in structure to $\phi_{c}$.
\end{proof}

\begin{lemma}\label{lem:sigma}
If $(\mathcal{R},\boxplus,\boxtimes, \zeta)$ is a $\boldsymbol{P\mathcal{R}}_{\zeta}$ and $a \in \mathcal{R}$. Then, the following maps are $\textbf{pro.con}$
\[ \sigma_{a} : \mathcal{R} \mapsto \mathcal{R} \] defined by $\sigma_{a}(w) = a \boxtimes w.$ 
\[ \gamma_{a} : \mathcal{R} \mapsto \mathcal{R}\] is defined by $\gamma_{a}(w) = w \boxtimes a.$
\end{lemma}

\begin{proof}
To show that $\sigma_{a}$ is $\textbf{pro.con}$. Consider $f_{a}: \mathcal{R} \mapsto \mathcal{R} \times \mathcal{R}$ where $f_{a}(w) = (a,w)$. Hence, $f_{a}$ is $\textbf{pro.con}$. In addition, the composition of the multiplicative map  with $f_{a}$ as $\operatorname{mul}_\zeta \circ f_{a}$ this yields that $\sigma_{a}$ is  $\textbf{pro.con}$.\newline
By the same reasoning as above, $\gamma_{a}$ is also  $\textbf{pro.con}$.
\end{proof}
In particular, the maps $\sigma_{a},\gamma_{a}$  turn out to be $\textbf{pro.homo}$. If
$a$ has a multiplicative inverse, as observed in Corollary ~\ref{Cor:homo}.

\begin{corollary}\label{Cor:homo}
Given that $(\mathcal{R},\boxplus,\boxtimes, \zeta)$ is a $\boldsymbol{P\mathcal{R}}_{\zeta}$ with unity and $b$ is the multiplicative inverse of $a$. Then, $\sigma_{a}$ and $\gamma_{a}$ are $\textbf{pro.homo}.$
\end{corollary} 
\begin{proof}
It can be observed that, $\sigma_{a^{-1}} = \sigma_{b}$ is $\textbf{pro.con}$. Moreover, 
\[ \sigma_a (\sigma_b (w)) = a \boxtimes b \boxtimes w = w \] where $a \boxtimes b = 1_\mathcal{R_{\zeta}}$.\newline\newline
In addition, 
\[ \sigma_b (\sigma_a (w)) = b \boxtimes a \boxtimes w = w .\] \newline
It can be seen that $\sigma_{a}$ is $\textbf{pro.homo}$. Similarly, $\gamma_{a}$ is $\textbf{pro.homo}$, as was shown for $\sigma_{a}$.
\end{proof}

\begin{theorem} \label{the:rightinvertible}
Assume that $(\mathcal{R}, \oplus, \otimes, \theta)$ is a $\boldsymbol{P\mathcal{R}}_{\theta}$ with unity. Let $\epsilon \in \mathcal{R}$ be a nonzero element and $\rho_{\epsilon} : \mathcal{R} \mapsto \mathcal{R}$ specified by $\rho_{\epsilon}(w)= \epsilon \otimes w$ is a $\textbf{pro.homo}$ map. Then, $\epsilon$ is a right invertible element.
\end{theorem}

\begin{proof}
Due to the fact that $\rho_{\epsilon}$ is $\textbf{pro.homo}$, I have  a $\textbf{pro.con}$ map $\hat\rho_{\epsilon} : \mathcal{R} \mapsto \mathcal{R}$ such that 
\[ \rho_{\epsilon} \circ \hat\rho_{\epsilon} (w) = \hat\rho_{\epsilon} \circ \rho_{\epsilon} (w) = w\]
Consider the left-hand side $\rho_{\epsilon} (\hat\rho _{\epsilon}(w))= w$. Hence,

\begin{align*}
\rho_{\epsilon} (\hat\rho _{\epsilon}(1_{\mathcal{R}_{\theta}})) &= 1_{\mathcal{R}_{\theta}} \\ \epsilon  \otimes \hat\rho _{\epsilon}(1_{\mathcal{R}_{\theta}}) &= 1_{\mathcal{R}_{\theta}}
\end{align*}
Since $\hat\rho _{\epsilon}(1_{\mathcal{R}_{\theta}})\neq 0$. Thus, $\epsilon$ is right invertible.  
\end{proof}

\begin{theorem} \label{the:leftinvetible}
For a $\boldsymbol{P\mathcal{R}}_{\theta}$ with unity, if $\lambda$ is a nonzero element of $\mathcal{R}$ and $\psi_{\lambda}: \mathcal{R} \mapsto \mathcal{R}, \psi_{\lambda}(k) = k \otimes \lambda $ is a $\textbf{pro.homo}$, implies that $\lambda$ is a left invertible element.
\end{theorem}
\begin{proof}
By a similar argument as in Theorem ~\ref{the:rightinvertible}.
\end{proof}
 If a nonzero element  $\eta$ satisfies the conditions of Theorems  [\ref{the:rightinvertible}  ,\ref{the:leftinvetible}]. Therefore, $\eta$ has a two-sided inverse and is invertible.

 \begin{proposition}
 Let $\mu : \mathcal{R} \mapsto \mathcal{R}$ be a map on $(\mathcal{R},\boxplus,\boxtimes,\zeta)$ defined by \[\mu(r)=w \boxtimes r \boxtimes k\] where $w , k \in \mathcal{R}$, provided that $\mu$ is $\textbf{pro.con}$.
 \end{proposition}
 
 \begin{proof}
In view of the continuity of $\sigma_{w}$ and $\gamma_k$ in Lemma ~\ref{lem:sigma}. The composition
\[\sigma_{w} \circ \gamma_{k}\] equals  $\mu$. Hence, $\mu$ is $\textbf{pro.con}$
 \end{proof}
 \begin{proposition}
 Suppose that $(\mathcal{R},\oplus,\odot,\theta)$ is $\boldsymbol{P\mathcal{R}}_{\theta}$. Define
 \[ 
 \mathcal{G}_{\theta} : \mathcal{R} \times \mathcal{R} \mapsto \mathcal{R}, \quad \mathcal{G} (w_{0},w_{1})=w_{1} \odot w_{0}\] 
 Then, $\mathcal{G}$ is $\textbf{pro.con}$

\begin{proof}
Define the map $\mathcal{F}_{\theta} : \mathcal{R}\times\mathcal{R} \longrightarrow \mathcal{R}\times\mathcal{R}$ by 
\[ 
\mathcal{F}_{\theta}(\hat w_{0},\hat w_{1}) = (\hat w_{1},\hat w_{0}).\]
To show that $\mathcal{F}_{\theta}$ is $\textbf{pro.con}$. Let 
\[(\mathcal{W}_{1}\times\mathcal{W}_{2}) \,\theta\, (\mathcal{K}_{1}\times\mathcal{K}_{2}), \quad {then} \quad 
\mathcal{W}_{1}\,\theta\,\mathcal{K}_{1} \quad {and} \quad \mathcal{W}_{2}\,\theta\,\mathcal{K}_{2}\] \newline
Since $\mathcal{F}_{\theta}(\mathcal{W}_{1}\times\mathcal{W}_{2})=\mathcal{W}_{2}\times\mathcal{W}_{1}
\quad\text{and}\quad
\mathcal{F}_{\theta}(\mathcal{K}_{1}\times\mathcal{K}_{2})=\mathcal{K}_{2}\times\mathcal{K}_{1}$ \newline
By the definition of the product proximity, I have
\[
(\mathcal{W}_{2}\times\mathcal{W}_{1}) \,\theta\, (\mathcal{K}_{2}\times\mathcal{K}_{1}).
\]
Therefore, $\mathcal{F}_{\theta}$ is $\textbf{pro.con}$.
Moreover, since $\mathcal{G}_{\theta} = \operatorname{mul}_\theta \circ \mathcal{F}_{\theta}$. It follows that, $\mathcal{G}_{\theta}$ is $\textbf{pro.con}$.
\end{proof}

 \end{proposition}
 \begin{proposition} \label{prop:closed}
 Let $(\mathcal{R},\boxplus,\boxtimes,\zeta)$ be a $\boldsymbol{P\mathcal{R}}_{\zeta}$, and $\mathcal{F} \subseteq \mathcal{R}$ be a closed set. Then, $\forall \varepsilon \in \mathcal{R}, \varepsilon \boxplus \mathcal{F}$ is also closed.
 \end{proposition}
 
 \begin{proof}
 Since $\mathcal{F}$ is a closed set, I have  $cl_{\zeta} \mathcal{F}$ $= \mathcal{F}$. To show that \[\varepsilon \boxplus \mathcal{F} = \{\varepsilon \boxplus \mathfrak{f} : \mathfrak{f} \in \mathcal{F}\}\]  is  closed, note that  $(\varepsilon \boxplus \mathcal{F}) \subseteq cl_{\zeta} (\varepsilon \boxplus \mathcal{F})$. Let $\beta_ {-\varepsilon}: \mathcal{R} \mapsto \mathcal{R}$ be defined by \[\beta_ {-\varepsilon}(w)= -\varepsilon \boxplus w\] which is a $\textbf{pro.homo}$ as in Lemma ~\ref{lem:boxplus}. Since any $\textbf{pro.con}$ map $\hat{h}$ satisfies  \[\hat{h}(cl_{\zeta} \mathcal{A}) \subseteq cl_{\zeta} (\hat{h}(\mathcal{A}))\] Applying this to $\beta_ {-\varepsilon}$ and  $\varepsilon \boxplus \mathcal{F}$ yields
 \begin{align*}
 \beta_ {-\varepsilon}(cl_{\zeta}(\varepsilon\boxplus \mathcal{F})) \subseteq cl_\zeta (\beta_ {-\varepsilon}(\varepsilon \boxplus \mathcal{F})) \subseteq cl_{\zeta}(\mathcal{F}) = \mathcal{F} \\
 \beta_\varepsilon(\beta_ {-\varepsilon}(cl_{\zeta}(\varepsilon\boxplus \mathcal{F}))) \subseteq \beta_\varepsilon(\mathcal{F}) \\
 cl_{\zeta}(\varepsilon \boxplus \mathcal{F}) \subseteq \varepsilon \boxplus \mathcal{F}
 \end{align*}
 Accordingly, $\varepsilon \boxplus \mathcal{F}$ is closed.  
 \end{proof}
 \begin{corollary}
For a  $\boldsymbol{P\mathcal{R}}_{\zeta}$. If $\mathcal{F}$ is a closed subset of $\mathcal{R}$, implies $\forall \varepsilon \in \mathcal{R}, \varepsilon \boxminus \mathcal{F}$ is also closed.
 \end{corollary}

 \begin{proof}
The result follows from Proposition ~\ref{prop:closed}. 
 \end{proof}
 \begin{proposition}
 Suppose that $(\mathcal{R},\boxplus,\boxtimes,\zeta)$ is a $\boldsymbol{P\mathcal{R}}_{\zeta}$ with unity. If $\lambda \in \mathcal{R}$ is an invertible element, and $\mathcal{V}$ is a closed subset of $\mathcal{R}$. Then, $\lambda \boxtimes \mathcal{V}$ and $\mathcal{V} \boxtimes \lambda$ are closed.
 \end{proposition}

 \begin{proof}
 Define \[ \lambda \boxtimes \mathcal{V} = \{\lambda \boxtimes  \mathfrak{v} : \mathfrak{v} \in \mathcal{V}\}\]
 Assume that $\alpha _{\lambda^{-1}} : \mathcal{R} \mapsto \mathcal{R}$ is a $\textbf{pro.homo}$ map as in Corollary ~\ref{Cor:homo}. It follows that
 \[ \alpha_{\lambda^{-1}} (cl_{\zeta} (\lambda \boxtimes \mathcal{V})) \subseteq cl_{\zeta}(\alpha_{\lambda^{-1}}(\lambda \boxtimes \mathcal{V})) \subseteq cl_{\zeta} (\mathcal{V}) = \mathcal{V} \]
  Applying  $\alpha_{\lambda}$ to  both sides yields
 \[ \alpha_{\lambda} (\alpha_{\lambda^{-1}} (cl_{\zeta} (\lambda \boxtimes \mathcal{V}))) \subseteq \alpha_{\lambda} (\mathcal{V})
 \]
 Consequently, \[cl_{\zeta}(\lambda \boxtimes \mathcal{V}) = \lambda \boxtimes \mathcal{V}\]
 which shows that $\lambda \boxtimes \mathcal{V}$ is closed. By a symmetric argument, $\mathcal{V} \boxtimes \lambda$ is also closed.
 \end{proof}
\begin{proposition}
Let $(\mathcal{R}_{1}, \boxplus, \boxtimes, \zeta)$ and $(\mathcal{R}_{2}, \oplus, \otimes, \theta)$ be two proximal rings. Then, their direct product $\mathcal{R}_{1}\times \mathcal{R}_{2}$ forms a proximal ring, denoted by $\boldsymbol{P\mathcal{R}}_{\zeta,\theta}$
\end{proposition}

\begin{proof}
To show that $\mathcal{R}_{1}\times \mathcal{R}_{2}$ is $\boldsymbol{P\mathcal{R}}_{\zeta,\theta}$. Define the map 
\[
\operatorname{add}_{\zeta,\theta} : (\mathcal{R}_{1} \times \mathcal{R}_{2} ) \times (\mathcal{R}_{1}\times \mathcal{R}_{2}) \mapsto (\mathcal{R}_{1} \times \mathcal{R}_{2} )\]
by \[
\operatorname{add}_{\zeta,\theta} ((w_{1},w_{2}),(\hat{w_{1}},\hat{w_{2}}))= (\operatorname{add}_{\zeta}(w_{1},w_{2}), \operatorname{add}_{\theta} (\hat{w_{1}},\hat{w_{2}}))\]
Since both coordinate maps $\operatorname{add}_{\zeta}$ and $\operatorname{add}_{\theta}$ are $\textbf{pro.con}$, their product map $\operatorname{add}_{\zeta,\theta}$ is $\textbf{pro.con}$ with respect to the product proximity $\zeta \times \theta$. \newline
By applying the same reasoning, one verifies that the $\operatorname{mul}_{\zeta,\theta}$  and $\operatorname{inv}_{\zeta,\theta}$ maps are \textbf{pro.con}. Therefore, the direct product $\mathcal{R}_{1} \times \mathcal{R}_{2} $ is itself a $\boldsymbol{P\mathcal{R}}_{\zeta,\theta}$.
\end{proof}
\begin{corollary}\label{cor:directproduct}
The direct product of a finite family of proximal rings, is a proximal ring.
\end{corollary}
\begin{proof}
The result follows directly by induction method.
\end{proof}
\begin{definition}
Let $(\mathcal{R}, \boxplus,\boxtimes)$ be a  field endowed with a proximity relation $\zeta$. Then, $(\mathcal{R},\boxplus, \boxtimes, \zeta)$ is called a proximal field,  denoted by $\boldsymbol{P\mathcal{F}}_{\zeta}$, if the maps $\operatorname{add}_\zeta$, $\operatorname{mul}_\zeta$ , and $\operatorname{inv}_\zeta$ are $\textbf{pro.con}$. Moreover, the inversion map
\[\colorbox{green!20}{\(\operatorname{inv^{-1}_{\zeta}} : \mathcal{R}^* \mapsto \mathcal{R}^*, \ \operatorname{inv^{-1}}(w) = w^{-1}\)}\]

is $\textbf{pro.con}$.
\end{definition}
\begin{example}
Let $(\mathbb{R},+,\cdot)$ be a field  equipped with the proximity relation $\zeta$, defined by 
\[ \mathcal{W} \;\; \zeta \;\; \mathcal{K} \;\;\iff\;\; D(\mathcal{W}, \;\mathcal{K}) = 0\]
where $D(\mathcal{W},\mathcal{K})= 
\left\{ \, \lvert \mathfrak{w} - \mathfrak{k} \rvert \;\middle|\; \mathfrak{w} \in \mathcal{W}, \;\mathfrak{k} \in \mathcal{K} \, \right\}$ denotes the Euclidean distance on $\mathbb{R}$. Then, $(\mathbb{R},+,\cdot,\zeta)$ is a $\boldsymbol{P\mathcal{F}}_{\zeta}$. \newline\newline
Since $\operatorname{add}_{\zeta}$, $\operatorname{mul}_{\zeta}$, and $\operatorname{inv}_{\zeta}$ are $\textbf{pro.con}$ on $\mathbb{R}$, it only remains to verify that the inversion map \[\operatorname{inv ^{-1}_{\zeta}} : \mathcal{R^{*}} \mapsto \mathcal{R^{*}}, \operatorname{inv ^{-1}}(w)=w^{-1}\]
\newline\newline
Is $\textbf{pro.con}$, consider $\mathcal{W}$ $\zeta$ $\mathcal{K}$, then $D(\mathcal{W},\mathcal{K}) = 0$. It follows that, there exist elements $\mathfrak{w} \in \mathcal{W}$ and $\mathfrak{k} \in \mathcal{K}$ such that \[\abs{\mathfrak{w} - \mathfrak{k}} = 0\] Hence, $\mathfrak{w} = \mathfrak{k}$, and therefore $\mathfrak{w^{-1}} = \mathfrak{k^{-1}}$. Consequently,
\[\mathcal{W}^{-1} \cap \mathcal{K}^{-1} \neq \phi\] This implies that, \
\[\operatorname{inv^{-1}_{\zeta}}(\mathcal{W}) \;\zeta\; \operatorname{inv^{-1}_{\zeta}}(\mathcal{K})\] 
Thus, $\operatorname{inv ^{-1}_{\zeta}}$ is \textbf{pro.con}. 
\end{example}
\begin{definition}
If $\mathcal{E}$ is an $\mathcal{R}-module$ and $(\mathcal{R},\boxplus,\boxtimes,\zeta)$ is a $\boldsymbol{P\mathcal{R}}_{\zeta}$. Then, $(\mathcal{E},\boxplus,\boxtimes,\zeta)$ is called a proximal $\mathcal{R}-module$, denoted by $\boldsymbol{P\mathcal{M}}_{\zeta}$, provided the following maps are $\textbf{pro.con}$:
\begin{enumerate}

\item [1.] $\operatorname{add}_\zeta ^{\mathcal{E}}: \mathcal{E} \times \mathcal{E} \mapsto \mathcal{E}, 
 \quad \boxed {\operatorname{add}_\zeta^{\mathcal{E}}(e_{0}, e_{1}) = e_{0} \boxplus e_{1}.} $
 
\item [2.] $\operatorname{mul}_\zeta ^{\mathcal{E}}: \mathcal{R} \times \mathcal{E} \mapsto \mathcal{E}, \quad  \boxed {\operatorname{mul}_\zeta ^{\mathcal{E}}(r, e) = r \boxtimes e. }$

\item [3.]$ \operatorname{inv}_\zeta ^{\mathcal{E}}: \mathcal{E} \mapsto \mathcal{E}, \quad \boxed {\operatorname{inv}_\zeta^{\mathcal{E}}(e) = -e.}$
\end{enumerate}
\end{definition}

The following lemma for proximal modules $\boldsymbol{P\mathcal{M}}_{\zeta}$ follows directly from Lemma \ref{lem:sigma} and Corollary \ref{Cor:homo} on proximal rings $\boldsymbol{P\mathcal{R}}_{\zeta}$.

\begin{lemma}
Let $\mathcal{E}$ be a $\boldsymbol{P\mathcal{M}}_{\zeta}$. Then, the following statements are hold:
\begin{enumerate}
\item[1.]  The map $\alpha _{\epsilon} : \mathcal{R} \mapsto \mathcal{E}$, defined by 
\[\alpha_{\epsilon}(w) = w \boxtimes \epsilon , \epsilon \in \mathcal{E}\]
is $\textbf{pro.con}$ \newline

\item[2.] $\forall \mathfrak{r} \in \mathcal{R}$, $\beta_{\mathfrak{r}} : \mathcal{E} \mapsto \mathcal{E}$ 
\[
\beta_{\mathfrak{r}}(e) = \mathfrak{r} \boxtimes e
\]
is $\textbf{pro.con}$. Moreover, if $\mathfrak{r}$ is invertible. Then, $\beta_{\mathfrak{r}}$ is $\textbf{pro.homo}$.
\end{enumerate}
\end{lemma}
\begin{proposition}
Let $\{\mathscr{A}_{j} : j \in J\}$ be a finite family of proximal $\mathcal{R}-modules$. Then, their direct product $\prod_{j \in J} \mathscr{A}_{j}$ is also a proximal $\mathcal{R}-module$.
\end{proposition}

\begin{proof}
The proof follows directly from Corollary \ref{cor:directproduct}.
\end{proof}
\subsection{Descriptive Rings}

\begin{definition}
Let $(\mathcal{R},\oplus,\otimes)$ be a ring with a descriptive relation $\zeta_{\Phi}$. Then $(\mathcal{R},\oplus,\otimes,\zeta_{\Phi})$ is called a descriptive  ring, denoted by $\boldsymbol{P\mathcal{R}}_{\zeta_{\Phi}}$, if the following maps are $\textbf{des.con}$  
 \\ For $(\hat w_{0}, \hat w_{1}) \in \mathcal{R} \times \mathcal{R}$ and $\hat w \in \mathcal{R}$, define:
\begin{enumerate}

\item [1.] $\operatorname{add}_{\zeta_{\Phi}}: \mathcal{R} \times \mathcal{R} \mapsto \mathcal{R}, 
 \quad \operatorname{add}_{\zeta_{\Phi}}(\hat w_{0}, \hat w_{1}) = \hat w_{0} \oplus \hat w_{1}. $
 
\item [2.] $\operatorname{mul}_{\zeta_{\Phi}}: \mathcal{R} \times \mathcal{R} \mapsto \mathcal{R}, \quad  \operatorname{mul}_{\zeta_{\Phi}}(\hat w_{0}, \hat w_{1}) = \hat w_{0} \otimes \hat w_{1}. $

\item [3.]$ \operatorname{inv}_{\zeta_{\Phi}}: \mathcal{R} \mapsto \mathcal{R}, \quad \operatorname{inv}_{\zeta_{\Phi}}(\hat w) = -\hat w.$
\end{enumerate}
\end{definition}

\begin{example}
Let
$\mathcal{K} = \{0, \text{rbA}, \text{rbB}, 1\}$ be a set of Ribbon complexes as illustrated in Figure \ref{fig:Ribbon complex}, where $\{0\} = rbA \cap rbB$ and $\{1\} = rbA \cup rbB$.\newline

\begin{figure}[h!]
    \centering

\includegraphics[width=0.8\textwidth]{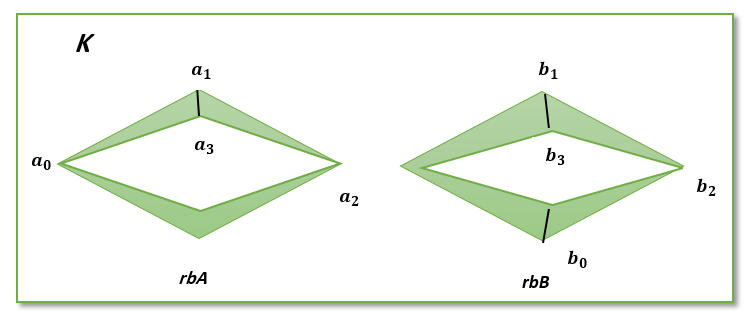} 
    \caption{Cell Complex $\mathcal{K}$.}
    \label{fig:Ribbon complex}
\end{figure}

Suppose the two binary operations $\oplus$ and $\odot$ are defined on $\mathcal{K}$ as follows:

\[
\begin{array}{c|cccc}
\oplus & 0 & \text{rbA} & \text{rbB} & 1 \\ \hline
0 & 0 & \text{rbA} & \text{rbB} & 1 \\
\text{rbA} & \text{rbA} & 0 & 1 & \text{rbB} \\
\text{rbB} & \text{rbB} & 1 & 0 & \text{rbA} \\
1 & 1 & \text{rbB} & \text{rbA} & 0
\end{array}
\qquad
\begin{array}{c|cccc}
\odot & 0 & \text{rbA} & \text{rbB} & 1 \\ \hline
0 & 0 & 0 & 0 & 0 \\
\text{rbA} & 0 & \text{rbA} & 0 & \text{rbA} \\
\text{rbB} & 0 & 0 & \text{rbB} & \text{rbB} \\
1 & 0 & \text{rbA} & \text{rbB} & 1
\end{array}
\]

where $\oplus$ denotes the \textbf{symmetric difference}, i.e.
\[
rbA \oplus rbB = (rbA \cup rbB) \setminus (rbA \cap rbB),
\]
and $\odot$ denotes the \textbf{intersection} between two Ribbon complexes.

\medskip

Now, define the descriptive  relation $\zeta_\Phi$ on $\mathcal{K}$ by
\[
A \;\zeta_\Phi\; B \quad \iff \quad \beta(A) = \beta(B),
\]
where
\[
\beta(\text{rbA}) = 2k + n \quad (\text{see \cite{PetersVergili:2020}})
\]
with $k$ the number of bridge edges and $n$ the number of common vertices.
Hence, $\beta(\text{rbA}) = 4$ and $\beta(\text{rbB}) = 5$, so that $rbA \;\underline{\zeta}_{\Phi}\; rbB$

Note that, $\operatorname{add}_{\zeta_{\Phi}}, \operatorname{mul}_{\zeta_{\Phi}}$ and $\operatorname{inv}_{\zeta_{\Phi}}$ are $\textbf{pro.con}$. Therefore, $(K,\oplus,\odot,\zeta_{\Phi})$ is $\boldsymbol{P\mathcal{R}}_{\zeta_{\Phi}}$
\end{example}

The proofs of the following results are analogous to those for the proximal rings.

\begin{lemma}
If $(\mathcal{R},\boxplus,\boxtimes,\zeta_{\Phi})$ is a $\boldsymbol{P\mathcal{R}}_{\zeta_{\Phi}}$ with unity, and $a \in \mathcal{R}$ is invertible, then the following maps are $\textbf{des.homo}$.

\begin{enumerate}
\item  $\phi_{a} : \mathcal{R} \mapsto \mathcal{R}, \phi_{a}(w)=w\boxplus a.$

\item $ \beta_{a} : \mathcal{R} \mapsto \mathcal{R}, \beta_{a}(w)=a\boxplus w.$ 

\item $\sigma_{a} : \mathcal{R} \mapsto \mathcal{R}, \sigma_{a}(w)=a \boxtimes w.$

\item $\gamma_{a} : \mathcal{R} \mapsto \mathcal{R}, \gamma_{a}(w)=w \boxtimes a.$
\end{enumerate}
\end{lemma}
\begin{proposition}
Let $(\mathcal{R},\boxplus,\boxtimes,\zeta_{\Phi})$ be a $\boldsymbol{P\mathcal{R}}_{\zeta_{\Phi}}$ with unity, If $\mathcal{F}$ is a closed subset of $\mathcal{R}$ and $c \in \mathcal{R}$ is invertible, then the following sets are closed.

\begin{enumerate}
\item[1.] $c \boxplus \mathcal{F}$.
\item[2.] $c \boxminus \mathcal{F}$.
\item[3.] $c \boxtimes \mathcal{F}$.
\item[4.] $\mathcal{F} \boxtimes c$.
\end{enumerate}
\end{proposition}

\begin{proposition}
The direct product of a finite family of descriptive rings, is a descriptive ring. 
\end{proposition}

\begin{definition}
 Let $(\mathcal{R}, \boxplus,\boxtimes)$ be a  field endowed with a descriptive relation $\zeta_{\Phi}$. Then, $(\mathcal{R},\boxplus, \boxtimes, \zeta_{\Phi})$ is called a descriptive field,  denoted by $\boldsymbol{P\mathcal{F}}_{\zeta_\Phi}$, if the maps $\operatorname{add}_{\zeta_{\Phi}}$, $\operatorname{mul}_{\zeta_{\Phi}}$ , and $\operatorname{inv}_{\zeta_{\Phi}}$ are $\textbf{des.con}$. Moreover, the inversion map
 
\[\colorbox{green!20}{\(\operatorname{inv^{-1}_{\zeta_\Phi}} : \mathcal{R}^* \to \mathcal{R}^*, \ \operatorname{inv}^{-1}_{\zeta_\Phi}(w) = w^{-1}\)}\]

is $\textbf{des.con}$.

\end{definition}
\begin{definition}
Let $\mathcal{M}$ be an $\mathcal{R}-module$ and $(\mathcal{R},\boxplus,\boxtimes,\zeta_{\Phi})$ is a $\boldsymbol{P\mathcal{R}}_{\zeta_{\Phi}}$. Then, $(\mathcal{M},\boxplus,\boxtimes,\zeta_{\Phi})$ is called a descriptive $\mathcal{R}-module$, denoted by $\boldsymbol{P\mathcal{M}}_{\zeta_{\Phi}}$, provided the following maps are $\textbf{des.con}$:
\newline
\begin{enumerate}

\item [1.] $\operatorname{add}_{\zeta_{\Phi}}^{\mathcal{M}}: \mathcal{M} \times \mathcal{M} \mapsto \mathcal{M}, 
 \quad \boxed {\operatorname{add}_{\zeta_{\Phi}}^{\mathcal{M}}(m_{0}, m_{1}) = m_{0} \boxplus m_{1}}.$ \newline
 
\item [2.] $\operatorname{mul}_{\zeta_{\Phi}}^{\mathcal{M}}: \mathcal{R} \times \mathcal{M} \mapsto \mathcal{M}, \quad  \boxed{\operatorname{mul}_{\zeta_{\Phi}}^{\mathcal{M}}(r, M) = r \boxtimes M.}$ \newline

\item [3.]$ \operatorname{inv}_{\zeta_{\Phi}}^{\mathcal{M}}: \mathcal{M} \mapsto \mathcal{M}, \quad \boxed{\operatorname{inv}_{\zeta_{\Phi}}^{\mathcal{M}}(m) = -m.}$
\end{enumerate}
\end{definition}
\section{Conclusion}

This work introduces the concepts of proximal rings and descriptive rings by combining the algebraic structure of rings with proximal and descriptive relations, based on which a topological system resulting from open subsets is defined. This study not only establishes a new topological algebraic framework, but also opens up promising avenues for future research.


\begin{thebibliography}{}

\bibitem{Cayley1854}
Cayley, A., ``On the theory of groups, as depending on the symbolic equation $\theta^n = 1$,'' \emph{Philosophical Magazine}, 1854.

\bibitem{Cech1966}
Čech, E., \emph{Topological Spaces}, John Wiley \& Sons Ltd., London, 1966; originally from seminar, Brno, 1936–1939; rev. ed. Z. Frolík, M. Katětov, MR0211373.

\bibitem{DiConcilio2018}
Di Concilio, C., Guadagni, C., Peters, J.F., and Ramanna, S., ``Descriptive proximities: Properties and interplay between classical proximities and overlap,'' \emph{Mathematics in Computer Science}, vol. 12, no. 1, pp. 91--106, 2018, MR3767897.

\bibitem{Efremovic1952}
Efremović, V.A., ``The geometry of proximity I (in Russian),'' \emph{Matematicheskii Sbornik, N.S.} 31(73), no. 1, pp. 189–200, 1952, MR0055659.

\bibitem{Hausdorff1914}
Hausdorff, F., \emph{Grundzüge der Mengenlehre}, Veit \& Comp., Leipzig, 1914.

\bibitem{Hilbert1897}
Hilbert, D., \emph{Die Theorie der algebraischen Zahlkörper}, 1897.

\bibitem{Is2024}
İs, Melih, ``Topological Group Construction in Proximity and Descriptive Proximity Spaces,'' \emph{Turkish Journal of Mathematics and Computer Science}, vol. 16, no. 1, pp. 206--216, 2024, DOI: 10.47000/tjmcs.1333562.

\bibitem{Lodato1962}
Lodato, M.W., ``On topologically induced generalized proximity relations,'' Ph.D. thesis, Department of Mathematics, Rutgers University, 1962, supervisor: S. Leader, MR0192470.

\bibitem{Peters2013}
Peters, F., ``Near sets: An introduction,'' \emph{Mathematics in Computer Science}, vol. 7, no. 1, pp. 3--9, 2013, DOI: 10.1007/s11786-013-0149-6, MR3043914.

\bibitem{Peters2014}
Peters, J.F., \emph{Topology of Digital Images: Visual Pattern Discovery in Proximity Spaces}, Intelligent Systems Reference Library 63, Springer, 2014, xv + 411 pp, Zentralblatt MATH Zbl 1295.68010.

\bibitem{PetersComputational2016}
Peters, F., \emph{Computational Proximity: Excursions in the Topology of Digital Images}, Intelligent Systems Reference Library 102, Springer, 2016, xxviii + 433 pp, DOI: 10.1007/978-3-319-30262-1.

\bibitem{PetersGuadagni2016}
Peters, F., and Guadagni, C., ``Strongly proximal continuity \& strong connectedness,'' \emph{Topology and its Applications}, vol. 204, pp. 41--50, 2016, MR3482701.

\bibitem{PetersOzturk2025}
Peters, F., and Ozturk, M.A., ``Near Stein-Weiss vector field groups in characteristic nearness approximation spaces in the complex plane,'' \emph{arXiv}, no. 18557v3, pp. 1–14, 2025, \url{https://doi.org/10.1016/j.topol.2016.02.008}.

\bibitem{PetersVergili:2020} 
Peters, J.F., and Vergili, T., \emph{Descriptive Fixed Set Properties for Ribbon Complexes}, arXiv preprint arXiv:2007.04394, 2020.

\bibitem{PetersVergili2023}
Peters, J.F., and Vergili, T., ``Good coverings of proximal Alexandrov spaces. Path cycles in the extension of the Mitsuishi-Yamaguchi good covering and Jordan Curve Theorems,'' \emph{Applied General Topology}, vol. 24, no. 1, pp. 25–45, 2023, \url{https://doi.org/10.4995/agt.2023.17046}.

\bibitem{Puisseux1850}
Puisseux, V., ``Recherches sur les fonctions algébriques,'' \emph{Journal de Mathématiques Pures et Appliquées}, vol. 15, pp. 365--480, 1850, MR0927683.

\bibitem{Warner1993}
Warner, S., \emph{Topological Rings}, Elsevier Science \& Technology, 1993, \url{https://books.google.ps/books?id=mZ3Wbewhpx0C}.

\bibitem{Chakravartty2023}
Chakravartty, P., Dutta, P., and Bandyopadhyay, S., \emph{A Brief Introduction and Some Methods of Construction of Topological Rings}, Master's Dissertation, Department of Pure Mathematics, University of Calcutta, 2023, \url{https://www.scribd.com/document/717275196/A-Brief-Introduction-and-Some-Methods-of-Construction-of-Topological-Rings?v=0.323}.

\end{thebibliography}
\end{document}